\documentclass[a4paper]{amsart}
\usepackage{amsmath} 
\usepackage{amsfonts} 
\usepackage{amssymb} 
\usepackage{amsthm} 
\usepackage{hyperref} 
\usepackage[utf8]{inputenc} 
\usepackage[british]{babel} 
\usepackage[protrusion=true,expansion=false]{microtype} 
\usepackage{mathtools} 
\mathtoolsset{showonlyrefs}
\usepackage[T1]{fontenc} 
\usepackage{dsfont} 
\usepackage{mathrsfs} 
\usepackage{xcolor} 
\usepackage{lmodern} 
\usepackage{enumitem} 
\usepackage{dsfont}
\usepackage{todonotes}
\usepackage{thmtools}
\usepackage[capitalise,noabbrev]{cleveref} 

\crefdefaultlabelformat{#2\textup{#1}#3}

\definecolor{dblue}{RGB}{0,0,178}
\definecolor{dgreen}{RGB}{31,183,41}
\definecolor{dred}{RGB}{174,0,0}
\definecolor{klimtpurple}{RGB}{158,52,158}
\definecolor{klimtblue}{RGB}{59,72,158} 
\definecolor{dklimtblue}{RGB}{20,25,158}
\definecolor{dklimtpurple}{RGB}{158,23,158}
\hypersetup{
	unicode=true,
	colorlinks=true,
	citecolor=dklimtpurple,
	linkcolor=dklimtblue,
	anchorcolor=dklimtblue,
	urlcolor=dklimtblue
}

\DeclareMathOperator{\Add}{Add}
\DeclareMathOperator{\Aut}{Aut}
\DeclareMathOperator{\cf}{cf}
\DeclareMathOperator{\dom}{dom}
\DeclareMathOperator{\hgt}{ht}
\DeclareMathOperator{\Ord}{Ord}
\DeclareMathOperator{\ot}{ot}
\DeclareMathOperator{\rank}{rk}
\DeclareMathOperator{\supp}{supp}

\newcommand*{\power}{\mathscr{P}}

\newcommand*{\ddc}{\dot{c}}
\newcommand*{\ddC}{\dot{C}}
\newcommand*{\ddF}{\dot{F}}
\newcommand*{\ddx}{\dot{x}}

\newcommand*{\bbP}{\mathbb{P}}
\newcommand*{\bbR}{\mathbb{R}}

\newcommand*{\calU}{\mathcal{U}}

\newcommand*{\afont}[1]{\mathsf{#1}}
\newcommand*{\AC}{\afont{AC}}
\newcommand*{\AW}{\afont{W}}
\newcommand*{\BPI}{\afont{BPI}}
\newcommand*{\DC}{\afont{DC}}
\newcommand*{\KWP}{\afont{KWP}}
\newcommand*{\PP}{\afont{PP}}
\newcommand*{\SVC}{\afont{SVC}}
\newcommand*{\UT}{\afont{CUT}}
\newcommand*{\WO}{\afont{WO}}
\newcommand*{\ZF}{\afont{ZF}}
\newcommand*{\ZFC}{\afont{ZFC}}

\newcommand*{\abs}[1]{\lvert#1\rvert}
\newcommand*{\Inj}[2]{{#2}^{\underline{\smash{#1}}}}
\newcommand*{\tup}[1]{\langle#1\rangle}

\newcommand*{\SetSymbol}[1][]{
\mathclose{}\nonscript\;#1|\nonscript\;\mathopen{}
}
\DeclarePairedDelimiterX{\Set}[1]{\{}{\}}{%
	\renewcommand\mid{\SetSymbol[\delimsize]}
	#1
}

\newcommand*{\1}{\mathds{1}}

\newcommand*{\concat}{{{}^{\frown}}}
\newcommand*{\defeq}{\mathrel{\vcenter{\baselineskip0.5ex \lineskiplimit0pt
                     \hbox{\scriptsize.}\hbox{\scriptsize.}}}%
                     =}
\newcommand*{\forces}{\mathrel{\Vdash}}
\newcommand*{\lomega}{{{<}\omega}}
\newcommand{\res}{\nobreak\mskip2mu\mathpunct{}\nonscript
  \mkern-\thinmuskip{\restriction}\mskip6muplus1mu\relax}
\newcommand*{\tmod}{\mathfrak{N}_{\aleph_1}}
\newcommand*{\varep}{\varepsilon}

\setlist[enumerate,1]{label=\textup{\arabic*.},ref=\textup{(\arabic*)}}
\setlist[enumerate,2]{label=\textup{(\roman*)},ref=\textup{(\roman*)}}

\newtheoremstyle{boldrk}
  {}{}
  {}{}
  {\bfseries}{.}
  {5pt plus 1pt minus 1pt}{}

\theoremstyle{plain}

\newtheorem{thm}{Theorem}[section]
\newtheorem{claim}{Claim}[thm]
\newtheorem{cor}[thm]{Corollary}

\newtheorem{prop}[thm]{Proposition}
\newtheorem*{fact}{Fact}
\newtheorem*{thmast}{Theorem}

\theoremstyle{boldrk}

\newtheorem{qn}[thm]{Question}

\newtheorem*{rk}{Remark} 

\theoremstyle{definition}

\newenvironment*{poc}[1][Proof of Claim]{\begin{proof}[#1]}{\end{proof}}

\Crefname{thm}{Theorem}{Theorems}
\Crefname{claim}{Claim}{Claims}
\Crefname{cor}{Corollary}{Corollaries}
\Crefname{lem}{Lemma}{Lemmas}
\Crefname{prop}{Proposition}{Propositions}
\Crefname{fact}{Fact}{Facts}
\Crefname{thmalph}{Theorem}{Theorems}
\Crefname{qn}{Question}{Questions}
\Crefname{defn}{Definition}{Definitions}
\Crefname{conj}{Conjecture}{Conjectures}

\title{Local reflections of choice}
\author{Calliope Ryan-Smith}

\email{c.Ryan-Smith@leeds.ac.uk}
\urladdr{https://academic.calliope.mx}
\address{School of Mathematics, University of Leeds, LS2 9JT, UK}

\date{\today{}}
\thanks{The author's work was financially supported by EPSRC via the Mathematical Sciences Doctoral Training Partnership [EP/W523860/1]. For the purpose of open access, the author has applied a Creative Commons Attribution (CC BY) licence to any Author Accepted Manuscript version arising from this submission. No data are associated with this article.}

\begin{document}

\begin{abstract}
Under the assumption of small violations of choice with seed \(S\) (\(\SVC(S)\)), the failure of many choice principles reflect to local properties of \(S\), which can be a helpful characterisation for preservation proofs. We demonstrate the reflections of \(\DC\), \(\AC_\lambda\), \(\PP\), and other important forms of choice. As a consequence, we show that if \(S\) is infinite then \(S\) can be partitioned into \(\omega\) many non-empty subsets.
\end{abstract}

\maketitle

\section{Introduction}

It is often the case that violating a consequence of choice is easier than verifying that a consequence of choice has been preserved. For example, to violate \(\AC_\omega\) in a symmetric extension, one need only add a countable family with no choice function. On the other hand, to ensure that \(\AC_\omega\) has been preserved one must check `every' countable family. Blass's \emph{small violations of choice} afford us an alternative approach. In any symmetric extension of a model of \(\ZFC\) there is a `seed' \(S\) such that \(\SVC(S)\) holds: For all non-empty \(X\) there is an ordinal \(\eta\) such that \(S\times\eta\) surjects onto \(X\) (see \cite[Theorem~4.3]{blass_injectivity_1979}).\footnote{Indeed \((\exists S)\SVC(S)\) has since been shown to be \emph{equivalent} to ``the universe is a set-generic symmetric extension of a model of \(\ZFC\)'', see \cite{usuba_choiceless_2021}.} In fact,  under this assumption, many violations of choice are reflected back and witnessed locally to \(S\). For example, to verify \(\AC_\omega\), one need only check \(\AC_\omega(S)\). Indeed, already in \cite{blass_injectivity_1979} the idea of a local reflection of choice was present, and the concept has appeared variously throughout the literature, summarised by the following. All notation will be introduced in the text as the results are proved.

\begin{thmast}
Assume \(\SVC(S)\) and \(\SVC^+(T)\).
\begin{enumerate}
\item (Blass, \cite{blass_injectivity_1979}) \(\AC\) is equivalent to ``\(S\) can be well-ordered''.
\item (Pincus--Blass, \cite{blass_injectivity_1979}) \(\BPI\) is equivalent to ``there is a fine ultrafilter on \([S]^\lomega\)''.
\item (Karagila--Schilhan, \cite{karagila_intermediate_2024}) \(\KWP_\alpha^\ast\) is equivalent to ``there is \(\eta\in\Ord\) such that \(\abs{S}\leq^\ast\abs{\power^\alpha(\eta)}\)''.
\item (Karagila--Schilhan, \cite{karagila_intermediate_2024}) \(\KWP_\alpha\) is equivalent to ``there is \(\eta\in\Ord\) such that \(\abs{T}\leq\abs{\power^\alpha(\eta)}\)''.
\end{enumerate}
\end{thmast}

Continuing this, we prove local equivalents of several common forms of choice, such as the principle of dependent choices and well-ordered choice. 

\begin{thmast}
Assume \(\SVC(S)\) and \(\SVC^+(T)\).
\begin{enumerate}
\item (\cref{prop:reflect-dc}) \(\DC_\lambda\) is equivalent to ``every \(\lambda\)-closed subtree of \(S^{{<}\lambda}\) has a maximal node or a chain of order type \(\lambda\)''.
\item (\cref{prop:reflect-ac-lambda}) \(\AC_\lambda\) is equivalent to \(\AC_\lambda(S)\), which is in turn equivalent to ``every function \(g\colon S\to\lambda\) splits''.
\item (\cref{cor:acwo-equiv}) \(\AC_\WO\) is equivalent to \(\AC_{{<}\aleph^\ast(S)}(S)\).
\item (\cref{prop:ac-x}) \(\AC_X\) is equivalent to \(\AC_X(S)\).
\item (\cref{prop:ut}) Assume that \(\cf(\omega_1)=\omega_1\). Then \(\UT\) is equivalent to \(\UT(T)\).
\item (\cref{prop:aw+}) \(\AW_\lambda\) is equivalent to \(\AW_\lambda(T)\).
\item (\cref{prop:aw+-ast}) \(\AW_\lambda^\ast\) is equivalent to \(\AW_\lambda^\ast(T)\).
\item (\cref{prop:pp}) \(\PP\) is equivalent to \(\PP\res T\land\AC_\WO\).
\item (\cref{prop:pp-not-res}) \(\PP(S)\land\AC_\WO\) implies \(\SVC^+(S)\). Hence, \(\PP\) is equivalent to \(\PP(S)\land\PP\res S\land\AC_\WO\).
\end{enumerate}
\end{thmast}

\subsection{Structure of the paper}

In \cref{sec:preliminaries} we go over some preliminaries for the paper. 
In \cref{sec:reflections} we describe reflections of various consequences of choice in the context of small violations of choice.

\section{Preliminaries}\label{sec:preliminaries}

We work in \(\ZF\). We denote the class of ordinals by \(\Ord\). Given a set of ordinals \(X\), we use \(\ot(X)\) to denote the order type of \(X\). Given an ordinal \(\alpha\), \(\cf(\alpha)\) is the cofinality of \(\alpha\) (the least cardinality of a cofinal subset). For sets \(X,Y\), \(\abs{X}\leq\abs{Y}\) means that there is an injection \(X\to Y\), \(\abs{X}\leq^\ast\abs{Y}\) means that \(X=\emptyset\) or there is a surjection \(Y\to X\),\footnote{Equivalently, there is a partial surjection \(Y\to X\).} and \(\abs{X}=\abs{Y}\) means that there is a bijection \(X\to Y\). For a well-orderable set \(X\), we use \(\abs{X}\) to mean \(\min\Set{\alpha\in\Ord\mid\abs{\alpha}=\abs{X}}\). By a \emph{cardinal} we mean a well-ordered cardinal, that is \(\alpha\in\Ord\) such that \(\abs{\alpha}=\alpha\). For a set \(Y\) and a cardinal \(\kappa\), we write \([Y]^\kappa\) to mean \(\Set{A\subseteq Y\mid\abs{A}=\abs{\kappa}}\) and \([Y]^{{<}\kappa}\) to mean \(\Set{A\subseteq Y\mid\abs{A}<\abs{\kappa}}\). Given a set \(X\), the \emph{Hartogs number} of \(X\) is \(\aleph(X)=\min\Set{\alpha\in\Ord\mid\abs{\alpha}\nleq\abs{X}}\). Dually, we define the \emph{Lindenbaum number} of \(X\) to be \(\aleph^\ast(X)=\min\Set{\alpha\in\Ord\mid\abs{\alpha}\nleq^\ast\abs{X}}\). It is a theorem of \(\ZF\) that \(\aleph(X)\) and \(\aleph^\ast(X)\) are well-defined cardinals. We denote concatenation of tuples by \(\concat\), so if \(f\colon\alpha\to X\) and \(g\colon\beta\to X\) then \(f\concat g\) is the function \(\alpha+\beta\to X\) given by
\begin{equation}
f\concat g(\gamma)=\begin{cases}
f(\gamma)&\gamma<\alpha\\
g(\delta)&\gamma=\alpha+\delta.
\end{cases}
\end{equation}

\subsection{Small violations of choice} Introduced in \cite{blass_injectivity_1979}, for a set \(S\) (known as the \emph{seed}), \(\SVC(S)\) is the statement ``for all \(X\) there is an ordinal \(\eta\) such that \(\abs{X}\leq^\ast\abs{S\times\eta}\)'', and \(\SVC\) is the statement \((\exists S)\SVC(S)\). We shall also make use of the injective form, \(\SVC^+(S)\) meaning ``for all \(X\), there is an ordinal \(\eta\) such that \(\abs{X}\leq\abs{S\times\eta}\)''. See \cite{ryan_hartogs_2024} for a more detailed overview of \(\SVC\) and \(\SVC^+\).

\begin{fact}
\(\SVC^+(S)\implies\SVC(S)\implies\SVC^+(\power(S))\).
\end{fact}

\subsection{Forcing and symmetric extensions}

While no knowledge of forcing or symmetric extensions is required for the main results, forcing is used in the proof of \cref{prop:pp-optimal,lem:pp-not-res-optimal}. For a thorough introduction to forcing and symmetric extensions one can go to \cite[Chapters~14 and 15]{jech_set_2003}, and for a more specific overview of the notation and terminology used in the proof of \cref{prop:pp-optimal,lem:pp-not-res-optimal}, one should see \cite{karagila_which_2024}. For a brief overview of the specific ideas used in this paper, continue reading this section.

A \emph{forcing} is a partial order \(\tup{\bbP,\leq}\) with maximal element \(\1\). We refer to the elements of \(\bbP\) as \emph{conditions}. We force downwards, so we say that \(q\) is \emph{stronger} than \(p\) (or \emph{extends} \(p\)) if \(q\leq p\). We say \(\bbP\) has the \emph{countable chain condition} (c.c.c.) if every antichain \(A\subseteq\bbP\) is countable. Given a set \(X\) of \(\bbP\)-names, we denote by \(X^\bullet\) the \(\bbP\)-name \(\Set{\tup{\1,\ddx}\mid\ddx\in X}\). By \(\Add(A,B)\) we mean the forcing with conditions that are partial functions \(p\colon B\times A\to 2\) such that \(\abs{\dom(p)}<\abs{A}\), and \(q\leq p\) when \(q\supseteq p\). For all \(B\), \(\Add(\omega,B)\) is c.c.c. Given a generic filter \(G\subseteq\Add(A,B)\), \(c=\bigcup G\) is a function \(B\times A\to2\), which we think of as encoding \(B\)-many functions \(A\to2\), so for \(b\in B\), the \(b\)\textup{th} Cohen subset of \(A\) added is \(\Set{a\in A\mid c(b,a)=1}\). This is formalised by the name
\begin{equation}
\ddc_b=\Set*{\tup{p,\check{a}}\mid a\in A,p\in\Add(A,B),p(b,a)=1}.
\end{equation}
In \cref{prop:pp-optimal} we consider the Feferman-style model \(\tmod\) from \cite{truss_models_1974}. This is given by letting \(G\) be an \(L\)-generic filter of \(\Add(\omega,\omega_1)\), and taking the least model \(\tmod\vDash\ZF\) such that \(L\cup\Set{\tup{c_\beta\mid\beta<\alpha}\mid\alpha<\omega_1}\subseteq M\). Here \(c_\beta\) is the \(\beta\)\textup{th} Cohen real \(\Set{n<\omega\mid\bigcup G(\beta,n)=1}\) as above. In \cite{truss_models_1974}, Truss shows that \(\tmod\) is a model of \(\AC_\WO\), \(V=L(w(\bbR))\), where \(w(\bbR)\) is well-orders of subsets of \(\bbR\), and that every subset of \(\bbR\) in \(\tmod\) is either well-orderable or contains a perfect subset.

In \cref{lem:pp-not-res-optimal} we consider Cohen's first model \(M\) from \cite{cohen_independence_1963}. This is given by letting \(G\) be \(L\)-generic for \(\Add(\omega,\omega)\) and taking the least model \(M\vDash\ZF\) such that \(L\cup\Set{c_n\mid n<\omega}\subseteq M\). Here \(c_n\) is again the \(n\)\textup{th} Cohen real. In \cite{cohen_independence_1963}, Cohen shows that \(M\) is of the form \(L(A)\), where \(A=\Set{c_n\mid n<\omega}\) is a Dedekind-finite set of reals. Additional information on Cohen's first model can be found in \cite[Sections~5.3 and 5.5]{jech_axiom_1973}.

\subsection{Form numbers}

Many of the forms of choice mentioned in this text are described and thoroughly examined for interdependence and equivalent statements in \cite{howard_consequences_1998} (to the extent that the tools at the time allowed). In particular, the consequences of the axiom of choice found are given numerical form numbers, which many still find helpful as a cataloguing tool. We therefore would like to remark the following form numbers of some of the subjects of this paper: \(\AC\) is Form~1; \(\BPI\) is Form~14; \(\UT\) is Form~31; \(\cf(\omega_1)=\omega_1\) is Form~34; \(\AC_\WO\) is Form~40; \(\AW_\lambda\) is Form~71(\(\alpha\)), where \(\lambda=\aleph_\alpha\); \(\KWP_n\) is Form~81(\(n\)), where \(n<\omega\); \(\AC_\lambda\) is Form~86(\(\alpha\)), where \(\lambda=\aleph_\alpha\); \(\DC_\lambda\) is Form~87(\(\alpha\)), where \(\lambda=\aleph_\alpha\); \(\PP\) is Form~101; \(\SVC\) is Form~191.

\section{Reflections}\label{sec:reflections}

\subsection{The principle of dependent choices}

A \emph{tree} is a partially ordered set \(\tup{T,\leq}\) with minimum element \(0_T\) such that, for all \(t\in T\), \(\Set{s\in T\mid s\leq t}\) is well-ordered by \(\leq\). This gives rise to a notion of rank \(\rank(t)=\sup\Set{\rank(s)+1\mid s<_Tt}\), of height \(\hgt(T)=\sup\Set{\rank(t)+1\mid t\in T}\), and of levels \(T_\alpha=\Set{t\in T\mid\rank(t)=\alpha}\). For \(x\in T_\alpha\) and \(\beta\leq\alpha\), we denote by \(x\res\beta\) the unique \(y\in T_\beta\) such that \(y\leq x\). A \emph{chain} is a set \(C\subseteq T\) such that for all \(s,t\in C\), \(s\leq t\) or \(t\leq s\). For an ordinal \(\alpha\), \(T\) is \emph{\(\alpha\)-closed} if every chain in \(T\) of order type less than \(\alpha\) has an upper bound. For an infinite cardinal \(\lambda\), \(\DC_\lambda\) is the statement ``every \(\lambda\)-closed tree has a maximal node or a chain of order type \(\lambda\)''.

We note that \(\DC_\lambda\) was originally defined in \cite{levy_interdependence_1964} as follows: Let \(X\) be a non-empty set and \(R\) a binary relation such that for all \(\alpha<\lambda\) and all \(\tup{x_\beta\mid\beta<\alpha}\in X^\alpha\) there is \(y\in X\) such that \(\tup{x_\beta\mid\beta<\alpha}\mathrel{R}y\). Then there is \(f\colon\lambda\to X\) such that, for all \(\alpha<\lambda\), \(f\res\alpha\mathrel{R}f(\alpha)\). This alternative definition is equivalent to the one that we are working with \cite[Theorem~1]{wolk_principle_1983}.

Given a set \(X\) and a limit ordinal \(\alpha\), we endow the set \(X^{{<}\alpha}=\bigcup\Set{X^\beta\mid\beta<\alpha}\) with a tree structure given by end-extension: \(f\leq g\) if and only if \(f\subseteq g\). A \emph{subtree} of \(X^{{<}\alpha}\) is non-empty \(T\subseteq X^{{<}\alpha}\) such that whenever \(f\leq g\) and \(g\in T\), \(f\in T\). In this case, \(f\res\beta\) meaning ``the function \(f\) restricted to domain \(\beta\)'' is the same as \(f\res\beta\) meaning ``the unique \(g\) of rank \(\beta\) such that \(g\leq f\)''.

\begin{prop}\label{prop:reflect-dc}
Assume \(\SVC(S)\). For each infinite cardinal \(\lambda\), the following are equivalent:
\begin{enumerate}
\item \(\DC_\lambda\);
\item every \(\lambda\)-closed subtree of \(S^{{<}\lambda}\) has a maximal node or a chain of order type \(\lambda\).
\end{enumerate}
\end{prop}

\begin{proof}
Certainly if \(\DC_\lambda\) holds then \(\DC_\lambda\) holds for subtrees of \(S^{{<}\lambda}\), so assume instead that every \(\lambda\)-closed subtree of \(S^{{<}\lambda}\) has a maximal node or a chain of order type \(\lambda\). Let \(T\) be a \(\lambda\)-closed tree, and let \(\eta\in\Ord\) and \(f\colon S\times\eta\to T\) be a surjection. We shall define a function \(\iota\colon S^{{<}\lambda}\to T\cup\Set{\bot}\), defining \(\iota(x)\) by induction on \(\dom(x)\), as follows: For \(x\in S^\alpha\),
\begin{enumerate}
\item let \(\iota(x)=\bot\) if there is \(\beta<\alpha\) such that \(\iota(x\res\beta)=\bot\); otherwise
\item if \(x=y\concat\tup{s}\), then let \(\iota(x)=f(s,\gamma)\), where \(\gamma\) is least such that \(f(s,\gamma)\) is an immediate successor of \(\iota(y)\); and
\item if \(x\) has limit rank, then let \(\iota(x)=\sup\Set{\iota(x\res\beta)\mid\beta<\alpha}\).
\end{enumerate}
Note that if \(\iota(x)\neq\bot\) then \(\iota(x)\in T_\alpha\), and that for all \(\beta<\alpha\), \(\iota(x\res\beta)=\iota(x)\res\beta\).

Let \(A=\Set{x\in S^{{<}\lambda}\mid\iota(x)\neq\bot}\).
\begin{claim}
\(A\) is a \(\lambda\)-closed subtree of \(S^{{<}\lambda}\).
\end{claim}
\begin{poc}
If \(x\in A\) and \(y<x\), then since \(\iota(x)\neq\bot\), \(\iota(y)\neq\bot\). Furthermore, \(\iota(\emptyset)=0_T\), so \(A\neq\emptyset\) and is indeed a subtree. If \(C\subseteq A\) is a chain of length less than \(\lambda\), then let \(b=\bigcup C\in S^{{\leq}\lambda}\). If \(b\in S^\lambda\) then \(\Set{\iota(b\res\alpha)\mid\alpha<\lambda}\) is a chain of order type \(\lambda\) in \(T\). So assume otherwise, that \(b\in S^{{<}\lambda}\). If \(b\) has limit rank \(\alpha\), say, then \(\iota(b)=\sup\Set{\iota(b\res\beta)\mid\beta<\alpha}\neq\bot\), since \(\iota(b\res\beta)\neq\bot\) for all \(\beta<\alpha\). Thus, \(b\in A\). If instead \(b\) has successor rank, then \(b\in C\), so certainly \(b\in A\) as required.
\end{poc}
If \(x\in A\) is a maximal node, then \(\iota(x)\) is maximal in \(T\): Otherwise, \(\iota(x)\) has an immediate successor, say \(t\), and \(t=f(s,\gamma)\) some \(\gamma\). Then \(\iota(x\concat\tup{s})\neq\bot\), and so \(x\) is not maximal in \(A\), contradicting our assumption. Finally, if \(A\) has no maximal nodes then there is a chain \(C\subseteq A\) of order type \(\lambda\). In this case, \(\iota``C\) is a chain in \(T\) of order type \(\lambda\) as required.
\end{proof}

\subsection{Well-ordered choice}

For a set \(X\), \(\AC_X\) is the statement ``if \(\emptyset\neq Y\), \(\emptyset\notin Y\), and \(\abs{Y}\leq^\ast\abs{X}\), then \(\prod Y\neq\emptyset\)'', where \(\prod Y=\Set{f\colon Y\to\bigcup Y\mid(\forall y\in Y)f(y)\in y}\) is the set of choice functions. \(\AC_\WO\) means \((\forall\alpha\in\Ord)\AC_\alpha\). For \(\alpha\in\Ord\), \(\AC_{{<}\alpha}\) means \((\forall\beta<\alpha)\AC_\beta\). \(\AC_X(A)\) is \(\AC_X\) for families of subsets of \(A\): If \(Y\subseteq\power(A)\setminus\Set{\emptyset}\) is non-empty and \(\abs{Y}\leq^\ast\abs{X}\) then \(\prod Y\neq\emptyset\). \(\AC_{{<}\alpha}(A)\) means \((\forall\beta<\alpha)\AC_\beta(A)\).

Given a function \(g\colon X\to Y\), we say that \(g\) \emph{splits} if there is a partial function \(h\colon Y\to X\) such that \(\dom(h)=g``X\) and \(gh(y)=y\) for all \(y\in g``X\).

We say that a set \(X\) is \emph{Dedekind-finite} if \(\abs{\omega}\nleq\abs{X}\). Otherwise, it is \emph{Dedekind-infinite}.

\begin{prop}\label{prop:reflect-ac-lambda}
Assume \(\SVC(S)\). For each infinite cardinal \(\lambda\), the following are equivalent:
\begin{enumerate}
\item\label{prop:ac-lambda;item:ac-lambda} \(\AC_\lambda\);
\item \(\AC_\lambda(S)\);
\item\label{prop:ac-lambda;item:splitting} every function \(g\colon S\to\lambda\) splits.
\end{enumerate}
\end{prop}

\begin{proof}
Certainly \(\AC_\lambda\) implies \(\AC_\lambda(S)\). Assuming \(\AC_\lambda(S)\), if \(g\colon S\to\lambda\) then we define
\begin{equation}
Y=\Set*{g^{-1}(\Set{\alpha})\mid\alpha\in g``S}.
\end{equation}
Then if \(c\in\prod Y\), \(h\colon g``S\to S\) given by \(h(\alpha)=c(g^{-1}(\alpha))\) splits \(g\).

Finally, assume that every \(g\colon S\to\lambda\) splits and let \(X=\Set{X_\alpha\mid\alpha<\lambda}\) be a collection of non-empty sets. Let \(f\colon S\times\eta\to\bigcup X\) be a surjection for some \(\eta\in\Ord\). For \(\alpha<\lambda\), let \(\beta_\alpha=\min\Set{\beta<\eta\mid (f``S\times\Set{\beta})\cap X_\alpha\neq\emptyset}\), and let \(S_\alpha=\Set{s\in S\mid f(s,\beta_\alpha)\in X_\alpha}\). Let \(g(s)=\min\Set{\alpha<\lambda\mid s\in S_\alpha}\) for all \(s\in S\). If \(h\colon\lambda\to S\) is a partial function splitting \(g\), then setting \(\gamma_\alpha=\min\Set{\gamma<\lambda\mid h(\gamma)\in S_\alpha}\) is well-defined, and \(C(\alpha)=f(h(\gamma_\alpha),\beta_\alpha)\) is a choice function for \(X\).
\end{proof}

The following corollary was also proved independently by Elliot Glazer (private communication).

\begin{cor}\label{cor:acwo-equiv}
Assume \(\SVC(S)\). The following are equivalent:
\begin{enumerate}
\item \(\AC_\WO\);
\item \(\AC_{{<}\aleph^\ast(S)}(S)\).
\end{enumerate}
\end{cor}

\begin{proof}
Certainly \(\AC_\WO\) implies \(\AC_{{<}\aleph^\ast(S)}(S)\), so assume \(\AC_{{<}\aleph^\ast(S)}(S)\). Let \(g\colon S\to\lambda\), \(A=g``S\), and \(\alpha=\ot(A)<\aleph^\ast(S)\). Taking \(\iota\colon A\to\alpha\) to be the unique isomorphism, we have \(\iota\circ g\colon S\to\alpha\). By \(\AC_{{<}\aleph^\ast(S)}(S)\), \(\iota\circ g\) is split, say by \(f\colon\alpha\to S\). Then \(f\circ\iota^{-1}\) splits \(g\). Since \(g\) was arbitrary, \cref{prop:reflect-ac-lambda} gives us \(\AC_\lambda\). Since \(\lambda\) was arbitrary, we obtain \(\AC_\WO\) as required.
\end{proof}

\begin{cor}\label{cor:linden-not-omega}
Assume \(\SVC(S)\). Then \(\aleph^\ast(S)\neq\aleph_0\).
\end{cor}

\begin{proof}
If \(\aleph^\ast(S)=\aleph_0\) then there are no surjections \(S\to\omega\), so every function \(S\to\omega\) has finite image and hence splits. Therefore, \(\AC_\omega\) holds. However, \(\AC_\omega\) implies that for all \(X\), \(\aleph^\ast(X)\neq\aleph_0\).
\end{proof}

We also have the following more direct (albeit longer) proof of this corollary.

\begin{proof}[Alternative proof]
Assume \(\SVC(S)\), where \(S\) is infinite (the case of \(S\) finite immediately gives \(\aleph^\ast(S)\neq\aleph_0\)), and let \(f\colon S\times\eta\to\Inj{\lomega}{S}\) be a surjection for minimal \(\eta\), where \(\Inj{n}{S}\) is the set of injections \(n\to S\) and \(\Inj{\lomega}{S}=\bigcup_{n<\omega}\Inj{n}{S}\). For \(s\in S\), let
\begin{equation}
X_s=\Set{\alpha<\eta\mid\tup{s,\alpha}\in\dom(f)\land(\forall\beta<\alpha)f(s,\beta)\neq f(s,\alpha)}
\end{equation}
and \(\eta_s=\ot(X_s)\). If \(\Set{\eta_s\mid s\in S}\) is infinite then \(\abs{\omega}\leq^\ast\abs{S}\) as required. Otherwise, \(\Set{\eta_s\mid s\in S}\) is finite. We consider two cases:

\textsl{Case}\enspace{}1: There is \(s\in S\) such that \(\eta_s\geq\omega\).\quad{}Then we can obtain an injection \(g\colon\eta_s\to S^{\underline{\lomega}}\) by setting \(g(\alpha)\) to be the \(\alpha\)\textup{th} element of \(X_s\). Hence \(S\) is Dedekind-infinite, and in particular \(\abs{\omega}\leq^\ast\abs{S}\).

\textsl{Case}\enspace{}2: Otherwise.\quad{}Then \(\eta_s\) is finite for all \(s\), and hence \(\eta\) is finite. We have \(\aleph^\ast(S\times\eta)\geq\aleph^\ast(S^{\underline{\lomega}})\geq\aleph_1\), but by the additivity of Lindenbaum numbers,\footnote{That is, if \(A\) is infinite then for all \(B\), \(\aleph^\ast(A\cup B)=\aleph^\ast(A)+\aleph^\ast(B)\). In particular, if \(n<\omega\) and \(A\) is infinite then \(\aleph^\ast(n\times A)=\aleph^\ast(A)\).} we have \(\aleph^\ast(S\times\eta)=\aleph^\ast(S)\), and hence \(\aleph^\ast(S)\geq\aleph_1\).
\end{proof}

The idea behind \cref{cor:linden-not-omega} extends to certain other cardinals \(\kappa\), though we additionally have to assume \(\AC_{{<}\kappa}\), as (unlike \(\AC_\lomega\)) it is not automatic.

\begin{prop}
Let \(\kappa\) be a limit cardinal or singular. Assume \(\SVC(S)\) and \(\AC_{{<}\kappa}\). Then \(\aleph^\ast(S)\neq\kappa\).
\end{prop}

\begin{proof}
If \(\aleph^\ast(S)=\kappa\) then, by \cref{cor:acwo-equiv}, \(\AC_\WO\) holds. However, by \cite[Theorem~3.4]{ryan_hartogs_2024}, \(\AC_\WO\) is equivalent to ``for all \(X\), \(\aleph^\ast(X)\) is a regular successor'', contradicting that \(\aleph^\ast(S)=\kappa\) is singular or a limit.
\end{proof}

In fact, the method of \cref{prop:reflect-ac-lambda} applies more generally.

\begin{prop}\label{prop:ac-x}
Assume \(\SVC(S)\). Then for all \(X\), the following are equivalent:
\begin{enumerate}
\item \(\AC_X\);
\item \(\AC_X(S)\).
\end{enumerate}
\end{prop}

\begin{proof}
Certainly \(\AC_X\) implies \(\AC_X(S)\). so assume \(\AC_X(S)\) and let \(p\colon X\to Y\) be a surjection. Let \(f\colon S\times\eta\to \bigcup Y\) be a surjection. For \(x\in X\), let
\begin{align}
\beta_x&=\min\Set*{\beta<\eta\mid (f``S\times\Set{\beta})\cap p(x)\neq\emptyset},\text{ and}\\
S_x&=\Set*{s\in S\mid f(s,\beta_x)\in p(x)}.
\end{align}
By \(\AC_X(S)\), we have \(c\in\prod\Set{S_x\mid x\in X}\), giving \(g(p(x))\defeq f(c(x),\beta_x)\in\prod Y\).
\end{proof}

\subsection{The countable union theorem}

For a set \(X\), we write \(\UT(X)\) to mean ``a countable union of countable subsets of \(X\) is countable'', and \(\UT\) to mean the countable union theorem \((\forall X)\UT(X)\).

\begin{prop}\label{prop:ut}
Assume \(\SVC^+(S)\) and \(\cf(\omega_1)=\omega_1\). The following are equivalent:
\begin{enumerate}
\item \(\UT\);
\item \(\UT(S)\).
\end{enumerate}
\end{prop}

\begin{proof}
Certainly \(\UT\) implies \(\UT(S)\), so assume \(\UT(S)\). Let \(\Set{A_n\mid n<\omega}\) be a countable family of countable sets, and let \(A=\bigcup_{n<\omega}A_n\), assuming without loss of generality that \(A\subseteq S\times\eta\) for minimal \(\eta\).

For \(n<\omega\), let \(z_n=\Set{\alpha<\eta\mid(\exists s\in S)\tup{s,\alpha}\in A_n}\), so \(\abs{z_n}\leq^\ast\abs{A_n}\leq\abs{\omega}\), and so \(\abs{z_n}\leq\abs{\omega}\). By minimality of \(\eta\) and \(\cf(\omega_1)=\omega_1\), we have \(\eta=\ot(\bigcup_{n<\omega}z_n)<\omega_1\). For \(n<\omega\), let \(B_n=\Set{s\in S\mid(\exists\alpha<\eta)\tup{s,\alpha}\in A_n}\), so \(B_n\subseteq S\) is countable for all \(n\). By \(\UT(S)\), \(B=\bigcup_{n<\omega}B_n\) is countable, and hence \(\abs{A}\leq\abs{B\times\eta}\leq\aleph_0\) as required.
\end{proof}

\begin{qn}\label{qn:ut}
Can \cref{prop:ut} be improved to ``\(\UT\) is equivalent to \(\UT(S)\)'' without assuming that \(\omega_1\) is regular? Since the singularity of \(\omega_1\) is already a violation of \(\UT\), this is equivalent to ``does \(\SVC^+(S)\) and \(\cf(\omega_1)=\omega\) imply \(\lnot\UT(S)\)?''.
\end{qn}

\subsection{The axiom of choice}

The following was remarked by Blass in \cite{blass_injectivity_1979}.

\begin{prop}[Blass]\label{prop:ac}
Assume \(\SVC(S)\). The following are equivalent:
\begin{enumerate}
\item\label{prop:ac;item:ac} \(\AC\);
\item\label{prop:ac;item:ac-s} \(S\) can be well-ordered.
\end{enumerate}
\end{prop}

\begin{proof}
Certainly \(\AC\) implies that \(S\) can be well-ordered. On the other hand, if \(S\) can be well-ordered and \(\abs{X}\leq^\ast\abs{S\times\eta}\) then \(\abs{X}\leq\abs{S\times\eta}\), so \(X\) can be well-ordered.
\end{proof}

\subsection{Comparability}

\(\AW_X\) is the statement ``for all \(Y\), \(\abs{X}\leq\abs{Y}\) or \(\abs{Y}\leq\abs{X}\)'' and \(\AW_X^\ast\) is the statement ``for all \(Y\), \(\abs{X}\leq^\ast\abs{Y}\) or \(\abs{Y}\leq^\ast\abs{X}\)''. We write \(\AW_X^{(\ast)}(B)\) to mean ``for all \(A\subseteq B\), \(\abs{X}\leq^{(\ast)}\abs{A}\) or \(\abs{A}\leq^{(\ast)}\abs{X}\)''. Note that ``every infinite set is Dedekind-infinite'' is equivalent to \(\AW_{\aleph_0}\).

\begin{prop}\label{prop:aw+}
Assume \(\SVC^+(S)\). For each infinite cardinal \(\lambda\), the following are equivalent:
\begin{enumerate}
\item\label{prop:aw+;item:aw} \(\AW_\lambda\);
\item\label{prop:aw+;item:aw-s} \(\AW_\lambda(S)\).
\end{enumerate}
\end{prop}

\begin{proof}
Certainly \(\AW_\lambda\) implies \(\AW_\lambda(S)\), so assume \(\AW_\lambda(S)\). Let \(X\subseteq S\times\eta\), and let \(A=\Set{s\in S\mid(\exists\alpha<\eta)\tup{s,\alpha}\in X}\). Note that \(\abs{A}\leq\abs{X}\), since \(s\mapsto\tup{s,\alpha_s}\) is an injection, where \(\alpha_s\) is least such that \(\tup{s,\alpha_s}\in X\). If \(\abs{A}\leq\abs{\lambda}\) then \(\abs{X}\leq\abs{A\times\eta}\) is well-orderable, so certainly \(\abs{\lambda}\leq\abs{X}\) or \(\abs{X}\leq\abs{\lambda}\). On the other hand, if \(\abs{A}\nleq\abs{\lambda}\) then \(\abs{\lambda}\leq\abs{A}\leq\abs{X}\) as required.
\end{proof}

Replacing \(\leq\) by \(\leq^\ast\) in the proof of \cref{prop:aw+}, we obtain \cref{prop:aw+-ast}.

\begin{prop}\label{prop:aw+-ast}
Assume \(\SVC^+(S)\). For each infinite cardinal \(\lambda\), the following are equivalent:
\begin{enumerate}
\item\label{prop:aw-ast;item:aw-ast} \(\AW^\ast_\lambda\);
\item\label{prop:aw-ast;item:aw-ast-s} \(\AW^\ast_\lambda(S)\).
\end{enumerate}
\end{prop}

\begin{qn}
As a consequence of \cref{prop:aw+,prop:aw+-ast}, assuming \(\SVC(S)\) (and hence \(\SVC^+(\power(S))\)), \(\AW_\lambda(\power(S))\) implies \(\AW_\lambda\), and \(\AW_\lambda^\ast(\power(S))\) implies \(\AW_\lambda^\ast\). Under the assumption of \(\SVC(S)\), can we obtain a `better' set \(X\) such that \(\AW_\lambda(X)\) implies \(\AW_\lambda\)? What about the \(\AW_\lambda^\ast\) case?
\end{qn}

\subsection{Boolean prime ideal theorem}

The Boolean prime ideal theorem \(\BPI\) is the statement ``every Boolean algebra has a prime ideal'', though it has many equivalent forms (see \cite[Form~14]{howard_consequences_1998}). In \cite{blass_injectivity_1979}, Blass presents the following local reflection of \(\BPI\) under the assumption of \(\SVC\), attributing the idea behind the proof to Pincus.

\begin{prop}[Pincus--Blass, \cite{blass_injectivity_1979}]
Assume \(\SVC(S)\). The following are equivalent:
\begin{enumerate}
\item \(\BPI\);
\item There is a fine ultrafilter on \([S]^\lomega\). That is, an ultrafilter \(\calU\) on \([S]^\lomega\) such that, for all \(s\in S\), \(\Set{a\in [S]^\lomega\mid s\in a}\in\calU\).
\end{enumerate}
\end{prop}

\subsection{Kinna--Wagner principles}

For a set \(X\), we define the iterated power set function \(\power^\alpha(X)\) by \(\power^\alpha(X)=\bigcup_{\beta<\alpha}\power^\beta(X)\) when \(\alpha\) is a limit ordinal, and in the successor case \(\power^{\alpha+1}(X)=\power(\power^\alpha(X))\). We also extend this notation to \(\Ord\), so \(\power(\Ord)\) is the class of all \emph{sets} of ordinals, \(\power^2(\Ord)\) is the class of all sets of sets of ordinals, et cetera. For an ordinal \(\alpha\), \(\KWP_\alpha\) means ``for all \(X\) there is an ordinal \(\eta\) such that \(\abs{X}\leq\abs{\power^\alpha(\eta)}\)'', and \(\KWP_\alpha^\ast\) means ``for all \(X\), there is an ordinal \(\eta\) such that \(\abs{X}\leq^\ast\abs{\power^\alpha(\eta)}\)''. The following observations, from \cite{karagila_intermediate_2024}, are consequences of the fact that, for all \(\alpha\), there is a definable surjection from \(\power^\alpha(\Ord)\) onto \(\power^\alpha(\Ord)\times\Ord\), and that \(\Ord\subseteq\power^\alpha(\Ord)\).

\begin{prop}[Karagila--Schilhan, \cite{karagila_intermediate_2024}]\label{prop:kwp-ast}
Assume \(\SVC(S)\). The following are equivalent:
\begin{enumerate}
\item \(\KWP_\alpha^\ast\);
\item There is \(\eta\in\Ord\) such that \(\abs{S}\leq^\ast\abs{\power^\alpha(\eta)}\).
\end{enumerate}
\end{prop}

\begin{prop}[Karagila--Schilhan, \cite{karagila_intermediate_2024}]\label{prop:kwp}
Assume \(\SVC^+(S)\). The following are equivalent:
\begin{enumerate}
\item \(\KWP_\alpha\);
\item There is \(\eta\in\Ord\) such that \(\abs{S}\leq\abs{\power^\alpha(\eta)}\).
\end{enumerate}
\end{prop}

\begin{rk}
Given that \(\KWP_0\) and \(\KWP_0^\ast\) are both equivalent to \(\AC\), \cref{prop:kwp-ast,prop:kwp} give new context to \cref{prop:ac}.
\end{rk}

\subsection{The partition principle}

The partition principle \(\PP\) says ``for all \(X\) and \(Y\), \(\abs{X}\leq\abs{Y}\) if and only if \(\abs{X}\leq^\ast\abs{Y}\)''. Note that the forward implication always holds. By \(\PP\res X\) we mean the partition principle for subsets of \(X\): If \(A,B\subseteq X\) and \(\abs{A}\leq^\ast\abs{B}\) then \(\abs{A}\leq\abs{B}\). We instead write \(\PP(X)\) to mean ``for all \(A\), if \(\abs{A}\leq^\ast\abs{X}\) then \(\abs{A}\leq\abs{X}\)''.

\begin{prop}\label{prop:pp}
Assume \(\SVC^+(S)\). The following are equivalent:
\begin{enumerate}
\item \(\PP\);
\item \(\PP\res S\) and \(\AC_\WO\).
\end{enumerate}
\end{prop}

\begin{proof}
Certainly \(\PP\) implies \(\PP\res S\), and \(\PP\) implies ``for all \(X\), \(\aleph(X)=\aleph^\ast(X)\)'', which is equivalent to \(\AC_\WO\).\footnote{See the first lemma of \cite[Theorem~7]{pelc_weak_1978}, proof of which is attributed to Pincus by the author. In fact, the statement of the lemma in \cite{pelc_weak_1978} is ``\(\PP\) implies \(\AC_\WO\)'', but the proof only uses the assumption \((\forall X)\aleph(X)=\aleph^\ast(X)\), and the converse direction is straightforward. The proof can also be found in \cite[Theorem~3.1]{ryan_hartogs_2024}.} So instead assume \(\PP\res S\land\AC_\WO\). Let \(A,B\subseteq S\times\eta\) be such that \(\abs{A}\leq^\ast\abs{B}\), witnessed by \(f\colon B\to A\). We treat \(f\) as a partial surjection \(f\colon S\times\eta\to A\). For \(\tup{t,\alpha}\in A\), let
\begin{equation}
\varep_{t,\alpha}=\min\Set*{\varep<\eta\mid(\exists s\in S)f(s,\varep)=\tup{t,\alpha}}.
\end{equation}
Let \(B^{\tup{t,\alpha}}=\Set{s\in S\mid f(s,\varep_{t,\alpha})=\tup{t,\alpha}}\), and \(B^\varep=\bigcup\Set{B^{\tup{t,\alpha}}\mid\varep_{t,\alpha}=\varep}\). Let \(E=\Set{\varep<\eta\mid(\exists\tup{t,\alpha}\in A)\varep_{t,\alpha}=\varep}=\Set{\varep<\eta\mid B^\varep\neq\emptyset}\). For each \(\varep\in E\), let \(A^\varep=\Set{\tup{t,\alpha}\mid\varep_{t,\alpha}=\varep}\). Then \(\abs{A^\varep}\leq^\ast\abs{B^\varep}\), witnessed by \(s\mapsto f(s,\varep)\). Hence, by assumption, \(\abs{A^\varep}\leq\abs{B^\varep}\) and \(I^\varep=\Set{\text{injections }A^\varep\to B^\varep}\neq\emptyset\). By \(\AC_\WO\), let \(c\colon E\to\bigcup\Set{I^\varep\mid\varep\in E}\) be a choice function. Then
\begin{alignat}{3}
g\colon&&A&\to S\times\eta\\
&&\tup{t,\alpha}&\mapsto\tup{c(\varep_{t,\alpha})(t,\alpha),\varep_{t,\alpha}}
\end{alignat}
is an injection. Furthermore, for each \(\tup{t,\alpha}\in A\), \(f(g(t,\alpha))\in A^{\varep_{t,\alpha}}\). In particular, \(f(g(t,\alpha))\) is defined and so \(g\) is in fact an injection \(A\to B\) as required.
\end{proof}

\begin{qn}\label{qn:pp-improvement}
Can \cref{prop:pp} be improved to ``\(\PP\) is equivalent to \(\PP\res S\)''? Equivalently, is \(\AC_\WO\) a consequence of \(\SVC^+(S)\land\PP\res S\)?
\end{qn}

While we do not know if \(\AC_\WO\) is unnecessary in \cref{prop:pp}, we cannot weaken the requirement of \(\SVC^+(S)\) to \(\SVC(S)\), as \cref{prop:pp-optimal} demonstrates.

\begin{prop}\label{prop:pp-optimal}
Let \(M\) be the Feferman-style model \(\tmod\) from \cite{truss_models_1974}. That is, for \(G\subseteq\Add(\omega,\omega_1)\) an \(L\)-generic filter, we set
\begin{equation}
M=L\left(\Set*{\tup{c_\beta\mid\beta<\alpha}\mid\alpha<\omega_1}\right),
\end{equation}
where \(c_\beta\) is the \(\beta\)\textup{th} Cohen real introduced by \(G\). Then
\begin{equation}
M\vDash\AC_\WO\land\PP\res\bbR\land\SVC(\bbR)\land\lnot\PP.
\end{equation}
\end{prop}

\begin{proof}
Firstly, by \cite[Lemma~2.4]{truss_models_1974}, \(M\vDash\AC_\WO\).

By \cite[Theorem~3.2]{truss_models_1974}, every set of reals in \(M\) can either be well-ordered or contains a perfect subset. If \(A,B\subseteq\bbR\) and \(\abs{A}\leq^\ast\abs{B}\) then: If \(B\) can be well-ordered, \(\abs{A}\leq\abs{B}\); and if \(B\) contains a perfect subset then \(\abs{A}\leq\abs{\bbR}\leq\abs{B}\). Hence \(\PP\res\bbR\) holds.

By \cite[Lemma~2.3]{truss_models_1974}, \(M\vDash V=L(w(\bbR))\), where \(w(\bbR)\) is the set of well-orders of subsets of \(\bbR\). We aim to show that \(w(\bbR)\subseteq L(\bbR)\) and so \(M\vDash V=L(\bbR)\). In particular, this will show that \(M\vDash\SVC(\bbR)\).\footnote{If \(A\) is transitive and \(M=V(A)\), where \(V\vDash\ZFC\) then \(M\vDash\SVC([A]^\lomega)\) (see \cite{blass_injectivity_1979}). However, \(\ZF\vdash\abs{[\bbR]^\lomega}=\abs{\bbR}\).} As a consequence of \cite[Theorem~3.1]{truss_models_1974}, a set \(X\) of reals in \(M\) is well-orderable if and only if \(X\subseteq L[\tup{c_\beta\mid\beta<\alpha}]\) for some \(\alpha<\omega_1\). Hence, any well-ordered sequence \(f\colon\gamma\to\bbR\) in \(M\) is in fact an element of \(L[\tup{c_\beta\mid\beta<\alpha}]\) for some \(\alpha\). Since \(\alpha<\omega_1\), we may encode the entire sequence \(\tup{c_\beta\mid\beta<\alpha}\) as a single real \(c\), and so \(f\in L[c]\) for some \(c\in\bbR\). Hence, \(w(\bbR)\subseteq L(\bbR)\subseteq L(w(\bbR))\), and so \(M\vDash V=L(\bbR)\).

It remains to show that \(M\vDash\lnot\PP\), which we shall do this by showing that there is no injection \([\bbR]^\omega\to\bbR\) (noting that \(\ZF\vdash\abs{[\bbR]^\omega}\leq^\ast\abs{\bbR}\)).

Suppose that \(F\colon[\bbR]^\omega\to\bbR\) is such an injection, and assume that it is \(L\)-definable. Let \(\bbP=\Add(\omega,\omega_1)\) and, for \(p\in\bbP\), we define the support of \(p\), \(\supp(p)\), to be \(\Set{\beta<\omega_1\mid(\exists n<\omega)\tup{\beta,n}\in\dom(q)}\in[\omega_1]^\lomega\). Since \(F\) is definable in \(L\), \(F\) has a \(\bbP\)-name \(\ddF\) such that for all \(\sigma\in\Aut(\bbP)\) (where \(\Aut(\bbP)\) is the automorphism group of \(\bbP\)), \(\sigma\ddF=\ddF\).\footnote{In fact \(\ddF=\sigma\ddF\) for all automorphisms \(\sigma\) of the Boolean completion of \(\bbP\).} For \(\beta<\omega_1\), we define
\begin{equation}
\ddc_\beta=\Set*{\tup{p,\check{n}}\mid p\in\bbP,n<\omega,p(\beta,n)=1},
\end{equation}
so \(\ddc_\beta\) is a name for \(c_\beta\). Given a permutation \(\pi\) of \(\omega_1\), define \(\hat{\pi}\in\Aut(\bbP)\) by \(\hat{\pi}p(\pi(\alpha),n)=p(\alpha,n)\) for all \(p\in\bbP\) and \(\tup{\alpha,n}\in\omega_1\times\omega\). Note that for such automorphisms, \(\hat{\pi}\ddc_\beta=\ddc_{\pi(\beta)}\). Let \(\ddC=\Set{\ddc_n\mid n<\omega}^\bullet\).\footnote{That is, \(\ddC=\Set{\tup{\1,\ddc_n}\mid n<\omega}\).} Then \(C=\ddC^G\in[\bbR]^\omega\cap M\). Let \(p_0\in\bbP\) be such that \(p_0\forces``\ddF\) is an injection \([\dot{\bbR}]^\omega\to\dot{\bbR}"\).

Suppose that for some \(\bbP\)-name \(\ddx\) and some \(p\leq p_0\), \(p\forces\ddF(\ddC)=\ddx\). Suppose also that for some \(q\leq p\) and some \(n<\omega\), \(q\forces\check{n}\in\ddx\).
\begin{claim}
\(q\res\supp(p)\times\omega\forces\check{n}\in\ddx\).
\end{claim}

\begin{poc}
We shall show that for all \(r\leq q\res\supp(p)\times\omega\), \(r\not\forces\check{n}\notin\ddx\). Let \(r\leq q\res\supp(p)\times\omega\) be arbitrary. Then there is a permutation \(\pi\) of \(\omega_1\) such that \(\pi``\omega=\omega\), \(\pi\) fixes \(\supp(p)\) pointwise, and \(\supp(r)\cap \pi``\supp(q)\subseteq\supp(p)\) (noting that \(\supp(p)\), \(\supp(q)\), and \(\supp(r)\) are all finite). Then \(\hat{\pi}p=p\), \(\hat{\pi}\ddC=\ddC\), \(\hat{\pi}\ddF=\ddF\), and so \(\hat{\pi}q\forces\ddF(\ddC)=\hat{\pi}\ddx\). However, since \(\hat{\pi}q\leq\hat{\pi}p=p\), \(\hat{\pi}q\forces\ddF(\ddC)=\ddx\) as well, and thus \(\hat{\pi}q\forces\hat{\pi}\ddx=\ddx\). Furthermore, \(\hat{\pi}q\) and \(r\) have a common extension (namely \(\hat{\pi}q\cup r\)), and so \(\hat{\pi}q\cup r\forces\check{n}\in\ddx\). Therefore, \(r\not\forces\check{n}\notin\ddx\). Since no \(r\leq q\res\supp(p)\times\omega\) forces \(\check{n}\notin\ddx\), we must have that \(q\res\supp(p)\times\omega\forces\check{n}\in\ddx\).
\end{poc}

By the claim, we may assume that \(\ddx\) is an \(\Add(\omega,\supp(p))\)-name.

Let \(\beta,\beta'\in\omega_1\setminus\supp(p)\) be such that \(\beta<\omega\) and \(\beta'\geq\omega\), and let \(\pi\) be the transposition \(\begin{pmatrix}\beta&\beta'\end{pmatrix}\). Then \(\hat{\pi}p=p\) and \(\hat{\pi}\ddx=\ddx\), so \(p\forces\ddF(\hat{\pi}\ddC)=\ddx\). However, \(p\forces\ddc_{\beta'}\in\hat{\pi}\ddC\setminus\ddC\), and so \(p\forces\hat{\pi}\ddC\neq\ddC\), contradicting that \(p\forces``\ddF\) is an injection''.

In the case that \(F\) requires a real parameter, say \(c\), we note that by the c.c.c.\ of \(\Add(\omega,\omega_1)\), \(c\) has an \(\Add(\omega,\alpha)\)-name for some \(\alpha<\omega_1\). By working in \(L[G\res\alpha]\) (where \(G\res\alpha=\Set{p\in G\mid\dom(p)\subseteq\alpha\times\omega}\)) rather than \(L\), and noting that the quotient of \(\Add(\omega,\omega_1)\) by \(\Add(\omega,\alpha)\) (the \(\Add(\omega,\alpha)\)-name for the `rest of the forcing', so \(\Set{\tup{p\res\alpha\times\omega,p}\mid p\in\Add(\omega,\omega_1)}\)) is isomorphic to \(\Add(\omega,\omega_1)\), the same result follows (using \(\Set{\ddc_\beta\mid\alpha\leq\beta<\alpha+\omega}^\bullet\) instead of \(C\)).
\end{proof}

Even though we cannot improve \cref{prop:pp} to \(\SVC(S)\) as written, we can if we additionally assume the stronger axiom \(\PP(S)\), rather than merely \(\PP\res S\).

\begin{prop}\label{prop:pp-not-res}
\(\SVC(S)\land\PP(S)\land\AC_\WO\) implies \(\SVC^+(S)\).
\end{prop}

\begin{proof}
Let \(A\) be a set. By \(\SVC(S)\) there is a surjection \(h\colon S\times\eta\to A\) for some ordinal \(\eta\). For \(a\in A\), let \(\alpha_a=\min\Set{\alpha<\eta\mid(\exists s\in S)h(s,\alpha)=a}\). For \(\alpha<\eta\), let \(A_\alpha=\Set{a\in A\mid\alpha_a=\alpha}\). Then \(s\mapsto h(s,\alpha)\) is a (partial) surjection \(S\to A_\alpha\) for all \(\alpha\), and so, by \(\PP(S)\), \(\abs{A_\alpha}\leq\abs{S}\). Using \(\AC_\WO\), we may simultaneously pick injections \(i_\alpha\colon A_\alpha\to S\) for all \(\alpha<\eta\). Then \(a\mapsto\tup{i_{\alpha_a}(a),\alpha_a}\) is an injection \(A\to S\times\eta\); indeed, if \(\tup{i_{\alpha_a}(a),\alpha_a}=\tup{i_{\alpha_b}(b),\alpha_b}\) then \(\alpha_a=\alpha_b\), so \(i_{\alpha_a}(a)=i_{\alpha_a}(b)\), which implies \(a=b\) as the \(i_\alpha\) are injective. \(A\) was arbitrary, so we conclude \(\SVC^+(S)\).
\end{proof}

\begin{cor}\label{cor:pp-not-res-pp-implication}
Assume \(\SVC(S)\). The following are equivalent:
\begin{enumerate}
\item \(\PP\);
\item \(\PP\res S\), \(\PP(S)\), and \(\AC_\WO\).
\end{enumerate}
\end{cor}

\begin{proof}
Certainly \(\PP\) implies each of \(\PP\res S\), \(\PP(S)\), and \(\AC_\WO\), so instead assume \(\PP\res S\), \(\PP(S)\) and \(\AC_\WO\). By \cref{prop:pp-not-res}, \(\SVC^+(S)\) holds, and so by \cref{prop:pp}, \(\PP\) holds.
\end{proof}

In \cref{lem:pp-not-res-optimal} below, we prove that ``\(\PP\res S\land\AC_\WO\)'' cannot be omitted from \cref{cor:pp-not-res-pp-implication}.

\begin{prop}\label{lem:pp-not-res-optimal}
Let \(M=L(A)\) be Cohen's first model, where \(L\subseteq M \subseteq L[G]\) for \(L\)-generic \(G\subseteq\Add(\omega,\omega)\). Then
\begin{equation}
M\vDash\SVC^+(\bbR)\land\PP(\bbR)\land\lnot\PP.
\end{equation}
\end{prop}

\begin{proof}
For an overview of Cohen's first model and the proof of \(\SVC^+(\bbR)\), see \cite[Sections~5.3 and 5.5]{jech_axiom_1973}. Within is also a proof that there is an infinite Dedekind-finite set of reals in \(M\), contradicting \(\AC_\WO\) (and hence, by \cref{prop:pp}, \(\PP\)).

The proof that \(M\vDash\PP(\bbR)\) is due to Elliot Glazer and Assaf Shani. Suppose that \(f\colon\bbR\to X\) is a surjection and, using \(\SVC^+(\bbR)\), assume that \(X\subseteq\bbR\times\eta\) for some minimal \(\eta\). In \(L[G]\), \(2^{\aleph_0}=\aleph_1\), and (since \(\eta\) is minimal) \(f\) induces a surjection \(\bbR\to\eta\). Therefore \(\eta<\omega_2\). In \(M\), \(\abs{\omega_1}\leq\abs{\bbR}\), and so
\begin{equation}
\abs{X}\leq\abs{\bbR\times\eta}\leq\abs{\bbR\times\omega_1}\leq\abs{\bbR\times\bbR}=\abs{\bbR}.\qedhere
\end{equation}
\end{proof}

Furthermore, the model \(\tmod\) from \cref{prop:pp-optimal} shows that we cannot omit the \(\PP\res S\) requirement from \cref{cor:pp-not-res-pp-implication}.

\begin{prop}\label{lem:pp-not-r-pp-implication-optimal}
Let \(M\) be the Feferman-style model \(\tmod\) from \cite{truss_models_1974} (and \cref{prop:pp-optimal}). Then
\begin{equation}
M\vDash\AC_\WO\land\PP(\power(\bbR))\land\SVC^+(\power(\bbR))\land\lnot\PP.
\end{equation}
\end{prop}

\begin{proof}
We already saw in \cref{prop:pp-optimal} that \(M\) is a model of \(\AC_\WO\land\SVC(\bbR)\land\lnot\PP\). By \(\SVC(\bbR)\), \(\SVC^+(\power(\bbR))\) holds.

The proof of \(\PP(\power(\bbR))\) is similar to \cref{lem:pp-not-res-optimal}. Suppose \(\abs{X}\leq^\ast\abs{\power(\bbR)}\), witnessed by \(f\). We may assume that \(X\subseteq\power(\bbR)\times\eta\) for minimal \(\eta\). Since the outer model \(L[G]\vDash\abs{\power(\bbR)}=\aleph_2\), and \(f\) induces a surjection \(\power(\bbR)^M\to\eta\), we must have that \(\eta<\omega_3\). In \(M\), \(\abs{\omega_2}\leq\abs{\power(\bbR)}\) and so
\begin{equation}
\abs{X}\leq\abs{\power(\bbR)\times\eta}\leq\abs{\power(\bbR)\times\omega_2}\leq\abs{\power(\bbR)^2}=\abs{\power(\bbR^2)}=\abs{\power(\bbR)}.\qedhere
\end{equation}
\end{proof}

\section{Acknowledgements}

The author would like to thank Asaf Karagila for feedback on early versions of this text, and for suggesting the model exhibited in \cref{prop:pp-optimal}; Elliot Glazer and Assaf Shani for sharing their proof that Cohen's first model satisfies \(\PP(\bbR)\); and the referee for providing helpful feedback that improved both the clarity and accuracy of this text.

\bibliographystyle{amsplain}
\bibliography{local-reflections-of-choice}

\providecommand{\bysame}{\leavevmode\hbox to3em{\hrulefill}\thinspace}
\providecommand{\MR}{\relax\ifhmode\unskip\space\fi MR }
\providecommand{\MRhref}[2]{%
  \href{http://www.ams.org/mathscinet-getitem?mr=#1}{#2}
}
\providecommand{\href}[2]{#2}
\begin{thebibliography}{10}

\bibitem{blass_injectivity_1979}
Andreas Blass, \emph{Injectivity, projectivity, and the axiom of choice},
  Trans. Amer. Math. Soc. \textbf{255} (1979), 31--59. \MR{542870}

\bibitem{cohen_independence_1963}
Paul Cohen, \emph{The independence of the continuum hypothesis}, Proc. Nat.
  Acad. Sci. U.S.A. \textbf{50} (1963), 1143--1148. \MR{157890}

\bibitem{howard_consequences_1998}
Paul Howard and Jean~E.\ Rubin, \emph{Consequences of the axiom of choice},
  Mathematical Surveys and Monographs, vol.~59, American Mathematical Society,
  Providence, RI, 1998, With 1 IBM-PC floppy disk (3.5 inch; WD). \MR{1637107}

\bibitem{jech_axiom_1973}
Thomas~J.\ Jech, \emph{The {A}xiom of {C}hoice}, Studies in Logic and the
  Foundations of Mathematics, vol. Vol. 75, North-Holland Publishing Co.,
  Amsterdam-London; American Elsevier Publishing Co., Inc., New York, 1973.
  \MR{396271}

\bibitem{jech_set_2003}
\bysame, \emph{Set theory}, third millennium ed., Springer Monographs in
  Mathematics, Springer-Verlag, Berlin, 2003, Revised and expanded.
  \MR{1940513}

\bibitem{karagila_which_2024}
Asaf Karagila and Calliope Ryan-Smith, \emph{Which pairs of cardinals can be
  {H}artogs and {L}indenbaum numbers of a set?}, Fund. Math. \textbf{267}
  (2024), no.~3, 231--241. \MR{4831492}

\bibitem{karagila_intermediate_2024}
Asaf Karagila and Jonathan Schilhan, \emph{Intermediate models and
  {K}inna--{W}agner principles}, ar{X}iv \textbf{2409.07352} (2024).

\bibitem{levy_interdependence_1964}
A.~L\'evy, \emph{The interdependence of certain consequences of the axiom of
  choice}, Fund. Math. \textbf{54} (1964), 135--157. \MR{162705}

\bibitem{pelc_weak_1978}
Andrzej Pelc, \emph{On some weak forms of the axiom of choice in set theory},
  Bull. Acad. Polon. Sci. S\'er. Sci. Math. Astronom. Phys. \textbf{26} (1978),
  no.~7, 585--589. \MR{515615}

\bibitem{ryan_hartogs_2024}
Calliope Ryan-Smith, \emph{The {H}artogs-{L}indenbaum spectrum of symmetric
  extensions}, MLQ Math. Log. Q. \textbf{70} (2024), no.~2, 210--223.
  \MR{4788363}

\bibitem{truss_models_1974}
John Truss, \emph{Models of set theory containing many perfect sets}, Ann.
  Math. Logic \textbf{7} (1974), 197--219. \MR{369068}

\bibitem{usuba_choiceless_2021}
Toshimichi Usuba, \emph{Choiceless {L}\"owenheim-{S}kolem property and uniform
  definability of grounds}, Advances in mathematical logic, Springer Proc.
  Math. Stat., vol. 369, Springer, Singapore, [2021] \copyright 2021,
  pp.~161--179. \MR{4378926}

\bibitem{wolk_principle_1983}
Elliot~S. Wolk, \emph{On the principle of dependent choices and some forms of
  {Z}orn's lemma}, Canad. Math. Bull. \textbf{26} (1983), no.~3, 365--367.
  \MR{703414}

\end{thebibliography}

\end{document}